\documentclass[a4paper,10pt]{article}

\usepackage{amsmath,amssymb,amsthm,a4wide}
\usepackage{setspace} 

\newtheorem{thm}{\bf Theorem}[section]
\newtheorem{Def}[thm]{\bf Definition}
\newtheorem{Lem}[thm]{\bf Lemma}
\newtheorem{Pro}[thm]{\bf Proposition}
\newtheorem{coro}[thm]{\bf Corollary}
\newtheorem{ex}{\bf Example}

\newcommand{\Ric}{{\text{Ric}}}
\newcommand{\R}{\mathbb{R}}
\newcommand{\N}{\mathbb{N}}
\newcommand{\s}{\mathbb{S}}
\newcommand{\D}{\mathbb{D}}
\newcommand{\Vol}{{\text{Vol}}}
\newcommand{\tphi}{{\tilde{\phi}}}
\newcommand{\diam}{{\text {diam}}}
\newcommand{\skel}{{\text {skel}}}
\newcommand{\del}{{\partial}}

\title{Volume growth and the topology of pointed Gromov-Hausdorff limits}
\author{M. Munn}

\begin{document}

\maketitle

\begin{abstract}

In this paper we examine topological properties of pointed metric measure spaces $(Y, p)$ that can be realized as the pointed Gromov-Hausdorff limit of a sequence of complete, Riemannian manifolds $\{(M^n_i, p_i)\}_{i=1}^{\infty}$ with nonnegative Ricci curvature. Cheeger and Colding \cite{ChCoI} showed that given such a sequence of Riemannian manifolds it is possible to define a measure $\nu$ on the limit space $(Y, p)$. In the current work, we generalize previous results of the author to examine the relationship between the topology of $(Y, p)$ and its volume growth. Namely, given constants $\alpha(k,n)$ which were computed in \cite{Munn} and based on earlier work of G. Perelman, we show that if $\lim_{r \to \infty} \frac{\nu(B_p(r))}{\omega_n r^n} > \alpha(k,n)$, then the $k$-th group of $(Y,p)$ is trivial. The constants $\alpha(k,n)$ are explicit and depend only on $n$, the dimension of the manifolds $\{(M^n_i, p_i)\}$, and $k$, the dimension of the homotopy in $(Y,p)$.
\end{abstract}

\maketitle

\section{Introduction}

The Gromov-Hausdorff limits of Riemannian manifolds with Ricci curvature lower bounds have become an important focus in the modern study of Reimannian geometry. In their joint work, Cheeger and Colding proved a number of substantial geometric properties and regularity results describing the nature of these limit spaces \cite{ChCoI, ChCoII, ChCoIII}. In particular, they show that the limit space $Y$ is in fact a metric space equipped with a measure $\nu$ satisfying the  Bishop-Gromov Volume Comparison Theorem originally stated for Riemannian manifolds (see Theorem \ref{Thm-classic Bishop Gromov}). That is, assuming nonnegative Ricci curvature in the sequence, for all $z \in Y$ and $0< r_1 \leq r_2$,
\begin{equation}\label{renormalized vol comparison}
\frac{\nu(B_z(r_1))}{\nu(B_z(r_2))} \geq \left(\frac{r_1}{r_2}\right)^n.
\end{equation}
Letting $\omega_n$ denote the volume of an $n$-dimensional unit ball in Euclidean space, it follows that $\lim_{r \to \infty} \dfrac{\nu(B_z(r))}{\omega_n r^n} \leq 1$. Here we study the topology of such limit spaces with an additional condition of Euclidean volume growth in the limit, that is $\lim_{r \to \infty} \dfrac{\nu(B_z(r))}{\omega_n r^n} > 0$. We show that the topology of the limit space simplifies tremendously when the volume growth of $Y$ is carefully restricted.

Generally speaking, the Gromov-Hausdorff distance defines a very weak form of convergence. While in dimensions 1 and 2 the topology is somewhat well behaved, in higher dimensions very little can be said even when additional geometric constraints are assumed for the sequence $M^n_i$. For example, in \cite{Men} Menguy constructs a sequence of 4-dimensional manifolds each with positive Ricci curvature and Euclidean volume growth but whose Gromov-Hausdorff limit possesses infinite topological type within balls of arbitrarily small radius. In the positive direction, Anderson \cite{And1} showed that the limit is in fact a $C^{1, \alpha}$ manifold assuming a two sided Ricci curvature bound and a uniform lower bound on the injectivity radius of the sequence. Later, Sormani-Wei \cite{SoWei2} showed that when the sequence has a uniform lower bound on the Ricci curvature the limit space admits a universal cover and, in fact, if the sequence is additionally simply connected then the limit space is its own universal cover. Further work of Ennis-Wei \cite{EnnWei} describes the nature of the universal cover when the limit space is compact. 

The results of Sormani-Wei \cite{SoWei2} can be also be used to extend another theorem of Anderson's \cite{And2} comparing the volume growth of a Riemannian manifold to the size of its fundamental group. They show that if the volume growth of the sequence is at least half that of Euclidean space then the universal cover of the limit is the space itself. However, this does not imply that the limit is simply connected. In fact, it remains an open question whether this condition on volume growth implies simply connectedness in the limit or not (see Remark 2).

Here we prove a partial result in this direction. We give precise bounds $\alpha(1,n)$, see Table \ref{Table-alphas}, for the volume growth of $M^n_i$ such that the following theorem holds

\begin{thm}\label{Thm--simply connected limit}
Let $\{(M^n_i, p_i)\}$ be a sequence of complete, $n$-dimensional Riemannian manifolds with $\Ric_{M_i} \geq 0$. If (passing to a subsequence if necessary)
\begin{equation}
 \lim_{r \to \infty} \frac{\Vol_{M_i}(B_{p_i}(r))}{\omega_n r^n} > \alpha(1,n),
\end{equation}
then the pointed Gromov-Hausdorff limit is simply connected.
\end{thm}

In general, one cannot assume that the Gromov-Hausdorff limit of a sequence of manifolds with $\pi_k(M^n_i) = 0$ has trivial $k$-th homotopy group. Taking capped cylinders and moving the base point to infinity through the Gromov-Hausdorff convergence produces a cylinder in the limit (see Example 3.\ref{Ex-limit cylinder}). Although the elements of this sequence are simply connected the limit is certainly not. However, Theorem \ref{Thm--simply connected limit} shows that this property does in fact hold provided the volume growth is sufficiently large for the manifolds throughout the sequence. 

In fact, Theorem \ref{Thm--simply connected limit} is a consequence of a more general theorem relating the volume growth of the limit space and {\em any} of its $k$-th homotopy groups. Furthermore, the actual dependence on volume growth occurs not on the $M^n_i$ in the sequence, but only in the volume growth of their Gromov-Hausdorff limit (as illustrated by Example 3.\ref{Ex-MenguyA}). We find constants $\alpha(k,n)$, depending only on the dimension of the maniofolds in the sequence and on $k$ the dimension of the homotopy, such that

\begin{thm}\label{Thm--Main limit theorem}
Let $(Y, p)$ be the pointed metric measure limit of a sequence of Riemannian manifolds $\{(M^n_i, p_i)\}$ all of whose Ricci curvatures is nonnegative, $\Ric_{M^n_i} \geq 0$, and let $\nu$ denote the renormalized limit measure of $Y$. If 
\begin{equation}
\lim_{r \to \infty} \frac{\nu(B_p(r))}{\omega_n r^n} > \alpha(k,n),
\end{equation}
then $\pi_k(Y, p) = 0$.
\end{thm}

\noindent Theorem \ref{Thm--simply connected limit} and Theorem \ref{Thm--Main limit theorem} are extensions of previous work of the author \cite{Munn} which in turn follow from a result of G. Perelman's on complete, Riemannian manifolds with nonnegative Ricci curvature \cite{Pe}. The proofs build upon Perelman's by determing precise constants for the volume growth where he only proved the existence of such a constant. These computations can be found in \cite{Munn}.

\textbf{Remark.}
As previously stated, Menguy \cite{Men} showed that even assuming positive Ricci curvature and Euclidean volume growth throughout the sequence it is possible for the Gromov-Hausdorff limit to have locally infinite topological type. In Theorem \ref{Thm--Main limit theorem} we allow nonnegative Ricci curvature in the sequence and only bound the volume growth in the limit. These conditions are enough to guarantee specific homotopy groups vanish in the limit and to control the limiting topology. In Example 3.\ref{Ex-MenguyA}, we adapt Menguy's construction to create a sequence where the volume growth restriction is only obtained in the limit. While the elements of the sequence do not have trivial topology, we find that the limit does. One caveat to Theorem \ref{Thm--Main limit theorem} however is that the bounds $\alpha(k,n)$ are very large and force the volume growth of $(Y,p)$ to be very near that of Euclidean space. It remains to show how sharp the values in Table \ref{Table-alphas} are and if our conclusions remain valid under less rigid conditions on $\alpha_Y$. Determining sharp bounds would more clearly illustrate the nature of the relationship between volume growth and topology for such Gromov-Hausdorff limit spaces.
\\

\textbf{Remark.}
In \cite {And2} Anderson shows that for a complete Riemannian manifold $M^n$ with nonnegative Ricci curvature the order of its fundamental group is bounded above by the reciprocal of its volume growth. In particular, if 
\begin{equation}
\lim_{r \to \infty} \frac{\Vol(B_p(r))}{\omega_n r^n} > \frac{1}{2},
\end{equation} 
then $\pi_1(M^n) = 0$. This was also proved independently by Li \cite{Li} using a heat kernel comparison argument. It is natural to ask whether such a uniform control on the volume growth of a sequence of Riemannian manifolds with $\Ric_{M^n_i} \geq 0$ forces simply connectedness in the limit as well. Anderson's proof utilizes the nonnegativity of Ricci curvature in the universal cover and his argument cannot be extended a priori to a Gromov-Hausdorff limit. The proof of Theorem \ref{Thm--Main limit theorem} however extracts information about the fundamental group without appealing to the deck transformations of the universal cover. It was shown by Sormani-Wei \cite{SoWei2} that the universal cover of the limit is the space itself but this alone does not imply simply connectedness, as illustrated in Example 3.\ref{Ex-Hawaii ring}. The spherical suspension over the Hawaiian earring is its own universal cover but it is not simply connected. Thus, in that sense Theorem \ref{Thm--simply connected limit} gives a partial solution to this problem.
\\

The condition on volume growth in Theorem \ref{Thm--Main limit theorem} implies the sequence $\{(M^n_i, p_i)\}$ is noncollapsing and thus (\cite{ChCoI}, Theorem 5.9) the limit measure $\nu$ is in fact a multiple of the $n$-dimensional Hausdorff measure ${\cal H}^n$ of $Y$. Therefore, restricting the volume growth of the manifolds throughout the sequence, yields precisely the same bound for the volume growth in the limit. We have

\begin{coro}\label{Corollary-topolgy of limit}
Let $\{(M^n_i, p_i)\}$ be a sequence of complete pointed Riemannian manifolds with $\Ric_{M_i} \geq 0$ converging in the pointed Gromov-Hausdorff sense to $(Y,p)$. If there exists a subsequence $\{(M_j, p_j)\}$ such that 
\begin{equation}
 \lim_{r \to \infty} \frac{\Vol_{M_j}(B_{p_j}(r))}{\omega_n r^n} > \alpha(k,n),
\end{equation}
then $\pi_k(Y,p) = 0$.
\end{coro}

Theorem \ref{Thm--Main limit theorem} can be seen as a natural generalization of results that appeared in \cite{Munn}. There we examined complete Riemannian manifolds $M^n$ with nonnegative Ricci curvature and employed techniques of Perelman to determine explicit values $\alpha(k,n)$ for bounds of the volume growth of $M^n$ (denoted by $\alpha_M$ and defined as in (\ref{volume growth of Y}) but replacing $\nu$ with the Riemannian volume element of $M^n$) which guarantee the $k$-th homotopy group of $M^n$ is trivial. These bounds depend only on $n$ the dimension of the manifold and $k$ the dimension of the homotopy. By construction, these constants increase with $k$ and thus knowing $\alpha_M > \alpha(k,n)$ actually implies $M^n$ is $k$-connected. Here we generalize this result to metric measure spaces, specifically those which are the pointed Gromov-Hausdorff limits of complete Riemannian manifolds satisfying

\begin{equation}\label{Ricci nonneg for sequence}
\Ric_{M_i} \geq 0.
\end{equation} 
Such a limit space $(Y,p)$ need not be a Riemannian manifold and so in some sense the condition of $\Ric \geq 0$ has its counterpart through (\ref{Ricci nonneg for sequence}). The requirement on the volume growth of $(Y,p)$ is equivalent to the Riemannian case only now applied to the renormalized limit measure on $Y$ obtained from \cite{ChCoI}. The difficulty arises in verifying that the primary tools that were used in \cite{Munn} also hold in an appropriate sense for the limit space $(Y,p)$. By proving analogs of these lemmas, we can yield the same conclusions for $\pi_k(Y,p)$ that we proved in the Riemannian setting.

To make these ideas more rigorous, we now briefly review the main ideas from \cite{Munn} and explain how these concepts can be adapted to examine metric measure limits. 

Let $M^n$ be a complete Riemannian manifold with $\Ric \geq 0$. The primary tool to show $\pi_k(M^n) = 0$ was the Homotopy Construction Theorem (\cite{Munn}, Theorem 2.7). This theorem states conditions which guarantee when a continuous map $f: \s^k \to M^n$ possesses a continuous extension $g$ on $\D^{k+1}$. The Homotopy Construction Theorem (HCT) is stated and proved for any complete locally compact length space and does not require the smooth structure of a Riemannian manifold. In \cite{Munn}, the necessary conditions of the HCT are shown to be satisfied when $M^n$ has $\Ric \geq 0$ and the volume growth obeys certain lower bounds. To obtain the conditions required to apply the HCT on $M^n$, we use two facts from the Riemannian geometry of manifolds with nonnegative Ricci curvature: the Abresch-Gromoll excess estimate (\cite{Munn}, Theorem 1.3; c.f. \cite{AbGr, Ch}) and a maximal volume lemma of Perelman (\cite{Munn}, Lemma; c.f. \cite{Pe}). Both these lemmas are generalized for the limit space $(Y,p)$ and appear here as Lemma \ref{Prop-General Perelman} and Lemma \ref{Prop-General AbGr} (resp.). The generalized Perelman Maximal Volume Lemma guarantees the existence of a well placed very long geodesic in $(Y,p)$ provided the volume growth is large enough. The generalization of the Abresch-Gromoll excess estimate is stated for geodesics in $(Y,p)$ which arise as the limit geodesics in the sequence of converging manifolds. Note that metric measure spaces with nonnegative Ricci curvature, as defined by Lott-Villani-Sturm \cite{LV,Sturm}, do not satisfy the Abresch-Gromoll inequality and thus our results here do not extend to that class of spaces. 

After proving these two generalizations, we then also have analogs of the Moving In Lemma (\cite{Munn}, Lemma 3.4) and the Main Lemma (\cite{Munn}, Lemma 3.5) as stated for the limit space $(Y,p)$. The proofs of the Moving In Lemma and Main Lemma follow precisely as in \cite{Munn} replacing the use of the Abresch-Gromoll excess estimate and Perelman's Maximal Volume Lemma with our generalized versions where necessary. The Moving In Lemma asserts that provided there is enough volume growth, given a continuous map $\phi: \s^k \to B_q(\rho)$ there exists another continuous map $\tphi$ also defined on $\s^k$ but whose image lies within a ball at $q$ with radius slightly smaller than $\rho$. The maps $\phi, \tphi$ need not be homotopic but the diameter of their images is controlled (in relation to the volume growth) in a uniform way. The Main Lemma provides a way of keeping track of the increase in volume growth as the dimension of the homotopy increases and is proven by induction. In this sense, the Moving In Lemma is the primary tool for constructing the homotopy (either in a Riemannian manifold $M^n$ or a limit space $(Y,p)$). In fact, this is the step in the argument where the volume growth restriction is introduced. With the Moving In Lemma in place for $(Y,p)$, the conditions for the HCT also hold in the limit space and thus the necessary homotopic extension exists to guarantee $\pi_k(Y,p) = 0$.

We proceed as follows: in Section \ref{Section-Generalizations-Background} we reveiw the basic ideas of pointed Gromov-Hausdorff distance and convergence and define the renormalized limit measure for the limits of sequences satisfying (\ref{Ricci nonneg for sequence}). In Section \ref{Section-Generalizations-Examples} we provide examples which aim to further demonstrate the relationship between the topology of pointed Gromov-Hausdorff limit spaces and their volume growth. In Section \ref{Section-Generalizations-Lemmas} we prove a generalization of the Abresch-Gromoll excess estmate (Lemma \ref{Prop-General AbGr}) and a generalization of Perelman's Maximal Volume Lemma (Lemma \ref{Prop-General Perelman}). In Section \ref{Section-Proof of limit theorem} we show how to prove Theorem \ref{Thm--Main limit theorem} using these generalized lemmas. 
\\

{\small \textbf{Acknowledgements.} This paper was completed as a Visiting Fellow at the University of Warwick under the NSF International Reserch Fellowship. I would like to thank the Mathematics Institute and the NSF for their support.}

\section{Background and Definitions} \label{Section-Generalizations-Background}

We begin by briefly discussing the notion of Gromov-Hausdorff distance and convergence, first for compact metric spaces and then for the noncompact case. The Gromov-Hausdorff distance defines a metric on the class of isometry classes of compact metric spaces, where the distance between isometric spaces is zero. More precisely,

\begin{Def}{\em (\cite{GLP}, Definition 3.4; c.f. \cite{BBI}, Definition 7.3.10).} Let $X$ and $Y$ be two compact metric spaces. The Gromov-Hausdorff distance between them, denoted $d_{GH}(X,Y)$, is defined as
$$d_{GH}(X,Y) = \inf d^Z_H(f(X), g(Y)),$$
\noindent where the infimum is taken over all metric spaces $Z$ and all isometric embeddings $f: X \to Z$, $g: Y \to Z$.

\end{Def}

\noindent Here $d^Z_H$ denotes the Hausdorff distance between subsets of $Z$ and is defined as
$$ d^Z_H(A,B) = \inf\{ \epsilon >0, : B \subset T_{\epsilon}(A) \textrm{and } A \subset T_{\epsilon}(B)\}, $$

\noindent where $T_{\epsilon}(A) = \{x \in Z : d_Z(x, A) < \epsilon \}$. The {\em pointed Gromov-Hausdorff distance} is defined exactly as above (as the infimum over Hausdorff distances of images in a common metric space $Z$) but with the additional requirement that $f(x) = g(y)$ in $Z$. 

A sequence of metric spaces $\{X_i\}^{\infty}_{i=1}$ converges in the Gromov-Hausdorff sense to a compact metric space $X$ provided $d_{GH}(X_i, X) \rightarrow 0$. When noncompact metric spaces are involved it is necessary to keep track of a sequence of points $p_i \in X_i$ through the convergence. We consider pointed metric spaces $\{(X_i, x_i)\}$ and define convergence in terms of the pointed Gromov-Hausdorff distance which is essentially convergence on compact sets. For compact metric spaces the pointed convergence and ordinary converge coincide. We define pointed Gromov-Hausdorff convergence as follows (c.f. Appendix in 
\cite{SoWei2})

\begin{Def}{\em (\cite{GLP}, Definition 3.14).}\label{Def-pted GH convergence} A sequence of noncompact metric spaces $(X_i, x_i)$ converges in the pointed Gromov Hausdorff sense to $(Y,y)$ if for all $R >0$ there exists a sequence $\epsilon_i \to 0$ such that $B_{x_i}(R + \epsilon_i)$ converges to $B_y(R)$ in the Gromov-Hausdorff sense.
\end{Def}

This is equivalent \cite{BBI} to the following: for every $r >0$ and $\eta >0$, there exists an $N>0$ such that for all $i >N$, there exists a (not necessarily continuous) map $f_i: B_{p_i}(r) \to X$ satisfying

1) $f_i(p_i) = p_{\infty}$;

2) $\sup_{x_1, x_2 \in B_{p_i}(r)}|d(f_i(x_1), f_i(x_2))- d(x_1, x_2)| < \eta$;

3) $T_{\epsilon}(f_i(B_{p_i}(r))) \supset B_{p_{\infty}}(r-\eta)$.

Such a map $f_i$ is called an {\it almost isometry} and if a sequence of manifolds converges in the Gromov-Hausdorff sense then a collection of almost isometries exists. Gromov-Hausdorff convergence defines a very weak topology for metric spaces; and in general, one can only expect that the limit of a sequence of length spaces is again a length space. Recall, a length space is a metric space where points can be connected by a length minimizing geodesic. However, not every path minimizing geodesic in the limit space is realized as the limit of geodesics in the sequence of manifolds.

\begin{Def}\label{def-limit geodesic}
Let $(Y,p)$ be the pointed Gromov-Hausdorff limit of a sequence of Riemannian manifolds $(M^n_i, p_i)$. A geodesic path $\sigma \in Y$ is called a limit geodesic if it can be realized as the limit of geodesics $\sigma_i \in M^n_i$; that is, there exists an almost isometry $f_i$ such that $f_i(\gamma_i(t)) = \gamma(t)$.
\end{Def}

When Gromov proved that the limit of a sequence of pointed Riemannian manifolds $\{(M^n_i, p_i)\}_{i=1}^{\infty}$ with lower bounded Ricci curvature was a length space $(Y,p)$, he in fact constructed limit geodesics (not just geodesics) between arbitrary pairs of points. Indeed, every pair of points in $Y$ has a limit geodesic of minimizing length connecting them. 

Recall the Bishop Gromov Volume Comparison Theorem for Riemannian manifolds
\begin{thm} {\em (\cite{BiCr, GLP})} \label{Thm-classic Bishop Gromov}
 Let $M^n$ be a complete $n$-dimensional Riemannian manifold with $\Ric \geq 0$. Then for all $0<r\leq R$,
\begin{equation} \label{formula-classic Bishop}
 \Vol(B_p(r)) \leq \omega_n r^n;
\end{equation}
and
\begin{equation}
 \dfrac{\Vol(B_p(r))}{\Vol(B_p(R))} \geq \left(\dfrac{r}{R} \right)^n.
\end{equation}
Equality holds in (\ref{formula-classic Bishop}) if and only if the ball $B_p(r) \subset M^n$ is isometric to the ball of radius $r$ in Euclidean space.
\end{thm}

Using a ball counting argument following from Theorem \ref{Thm-classic Bishop Gromov}, Gromov showed that 
\begin{thm} {\em (\cite{GLP}, c.f. \cite{BBI}, Theorem )}
Any pointed sequence of complete $n$-dimensional Riemannian manifolds $\{(M^n_i, p_i)\}$ with nonnegative Ricci curvature has a subsequence which converges in the pointed Gromov-Hausdorff topology to a pointed length
space $(Y, p)$.
\end{thm}

In \cite{ChCoI}, Cheeger-Colding examine the structure of spaces $Y$, which can be realized as the pointed Gromov-Hausdorff limits of sequences of complete, connected Riemannian manifolds, $\{(M^n_i, p_i) \}_{i=1}^{\infty}$ satisfing (\ref{Ricci nonneg for sequence}), see also \cite{Fu}. Among other things, they construct renormalized limit measures $\nu$ on the limit space $Y$ and show that such a measure satisfies an analog of the Bishop-Gromov volume comparison. 

\begin{thm}\label{Theorem-ChCo limit measure}{\em (\cite{ChCoI}, Theorem 1.6).}
Given any sequence of pointed Riemannian manifolds $\{(M^n_i, p_i) \}$ for which $\Ric_{M_i} \geq 0$, there is a subsequence $\{(M^n_j, p_j)\}$ convergent to some $(Y^m, p)$ in the pointed Gromov-Hausdorff sense, and a continuous function $\nu : Y^m \times \R^+ \rightarrow \R^+$, such that if $q_j \in M^n_j$, $z \in Y^m$, and $q_j \rightarrow z$, then for all $R>0$,
\begin{equation}
\frac{\Vol_{M_j}(B_{q_j}(R))}{\Vol_{M_j}(B_{p_j}(1))} \rightarrow \nu(B_z(R)).
\end{equation}
Furthermore, for all $z \in Y^m$ and $0 <  r_1 \leq r_2$, the renormalized limit measure $\nu$ satisfies the Bishop-Gromov type volume comparison stated in (\ref{renormalized vol comparison}).
\end{thm}

For $y \in Y$, the volume ratio $\frac{\nu(B_y(r))}{\omega_n r^n}$ is nonincreasing as a function of $r$. If in addition, 
$$
\Vol(B_{p_i}(1)) \geq v > 0,
$$ then we say that the sequence is {\em noncollapsing}. Otherwise, the sequence is said to {\em collapse}. Note that if a sequence of balls at the basepoints are noncollapsing then the same is true for any sequence of balls centered at other basepoints. This follows by applying the Bishop-Gromov Volume Comparison to the second sequence of balls and using a volume comparison argument to bound their volume from below. 

For any sequence, collapsed or not, it is possible to find a subsequence for which the renormalized limit measure exists. Note that in the noncollapsed case it is not necessary to pass to a subsequence nor renormalize the measure; the limit measure is simply the $n$-dimensional Hausdorff measure on $Y^m$. For any $R>0$, and $q_i \to q \in Y$,

\begin{equation}
\lim_{i \to \infty} \Vol(B_{q_i}(R)) = {\cal H}^n (B_q(R)),
\end{equation}

\noindent A renormalized limit measure is then a multiple of the $n$-dimensional Hausdorff measure.

Note, however that this uniqueness need not hold in the collapsed case and the renormalized limit measure can depend on the subsequence (see Example 1.24 in \cite{ChCoI}). Since we are concerned primarily with sequences satisfying the Euclidean volume growth condition in the limit, we will focus only on noncollapsed sequences. 

With this notion of measure for the limiting space $Y$, we can generalize the notion of volume growth to the class of metric measure limit spaces defined in Theorem \ref{Theorem-ChCo limit measure}

\begin{Def}\label{def-vol growth for mms limit}
Let $(Y,p)$ be the pointed Gromov-Hausdorff limit of a sequence $\{(M^n_i, p_i) \}$ of complete, connected Riemannian manifolds all of whose Ricci curvatures are nonnegative, $\Ric_{M_i} \geq 0$. Let $\nu$ denote the renormalized limit measure of $(Y,p)$ as defined above. Set
\begin{equation}\label{volume growth of Y}
\alpha_Y := \lim_{r \rightarrow \infty} \frac{\nu(B_p(r))}{\omega_n r^n}.
\end{equation}
\end{Def}

Note that $\alpha_Y$ is a global geometric invariant of $Y$ and it is independent of the base point $p$ in the definition. 

\section{Examples}\label{Section-Generalizations-Examples}

In this section we give some examples of pointed Gromov-Hausdorff limits to aid the reader with intuition and to further describe how the topology of the limit space is influenced by the volume growth of the limit and the nonnegativity of the Ricci curvature in the sequence. These examples are stated in 2 or 3 dimensions but many can be generalized to higher dimensions.

As stated before, the Gromov-Hausdorff metric gives a very weak notion of convergence and the topology can change in the limit even in 2 dimensions. Recall the following two well known examples:

\begin{ex}\label{Ex-limit cylinder}
 (Appearance of topology in the limit)
\end{ex}
\noindent There is a sequence of complete noncompact Riemannian manifolds $M^n_i$ with $\Ric \geq 0$ and $\pi(M^n_i) = 0$ whose Gromov-Hausdorff limit $Y$ is not simply connected.
\\

Let $M^2$ be the infinite half cylinder $\s^1 \times [0, \infty)$ capped off on one end with the upper hemisphere of $\s^2$ glued to $\s^1 \times \{0\}$ and the metric suitably smoothed at the union. $M^2$ is simply connected and has nonnegative Ricci curvature (in fact, nonnegative sectional curvature). Take a sequence of points $p_i \in M^2$ so that $d_M(N, p_i) \to \infty$ as $i \to \infty$ and where $N$ is the north pole of the hemisphere. The pointed sequence $(M^2, p_i)$ converges in the pointed Gromov-Hausdorff limit to $\s^1 \times \mathbb{R}$ which is not simply connected. 
\\

\begin{ex}\label{Ex-limit cone}
 (Disappearance of topology in the limit)
\end{ex}
\noindent There is a sequence of complete noncompact Riemannian manifolds $M^n_i$ with $\alpha_{M_i} > 0$ which are not simply connected whose limit space $Y$ is simply connected.
\\

Consider cones $M^2 = (\R^2, dt^2 + a^2 t^2 d\theta^2)$, with $0< a <1$  and the metric smoothed appropriately at the vertex. Attach a small handle  and fix a point $p$ near the vertex but away from the handle. By altering the metric through the sequence we can make the handle slide off to infinity away from the fixed point $p$. Note that each element of the sequence $(M^2_i, p)$ has Euclidean volume growth and the pointed Gromov-Hausdorff limit $(Y,p)$ is precisely a cone with $\alpha_Y = a >0$ and it is simply connected.
\\

The above examples show that the nature of the limit is influenced by the behavior of the base point through the convergence. Note however that the volume growth does not change through the convergence. Example \ref{Ex-limit cylinder} has linear volume growth throughout while Example \ref{Ex-limit cone} has Euclidean volume growth. It is possible for the volume growth to change as the sequence converges and, in fact, Theorem \ref{Thm--Main limit theorem} only requires the volume growth lower bound for the limit space. 

\begin{ex} (Dependence of volume growth bounds only for the limit) \label{Ex-MenguyA}
\end{ex}
\noindent There is a sequence of complete, noncompact 4-manifolds $M^4_i$ with $\Ric > 0$, Euclidean volume growth satisfying $0< \alpha_{M^4_i} \leq \alpha(2,4)$ and with $\pi_2(M^4_i) \neq 0$ whose limit space has volume growth $\alpha_Y > \alpha(2,4)$ and $\pi_2(Y) = 0$.
\\

In \cite{Men}, Menguy constructs a 4-dimensional manifold with positive Ricci curvature and Euclidean volume growth (i.e. $\alpha_{M^4_i} \geq c > 0$) with infinite topological type. The construction begins by defining a metric of metric cone over a spherical suspension over a small ball (see also \cite{ShSo}). The result is a double warped product
$$
g_{M^4} = dr^2 + (cr)^3(dt^2 + \sin^2 t \cdot R_0 d\sigma^2),
$$
where $d\sigma$ is the metric on the round sphere, $0< c < 1$ and $R_0 < 1$. 

The cone structure ensures the manifold has Euclidean volume growth like $cr^4$. Menguy then glues in a building block of Perelman \cite{Pe2} which has nontrivial topology along the edge of the cone formed from the singular points of the suspension. The metric can be smoothed to ensure the Ricci curvature is always positive. The final product is a manifold with positive Ricci curvature and Euclidean volume growth but nontrivial $\pi_2$. From \cite{Munn} we can say more about the volume growth of this $M^4$. Namely, it must have $\alpha_{M^4_i} \leq \alpha(2,4)$.

Let $p$ be the vertex of the cone in $M^4_i$ and consider the pointed Gromov-Hausdorff limit achieved by blowing up the metric at this point. That is take $(M^4_i, p)$ with the metric 
$$
g_{M^4_i} = r_i^{-2}g_{M^4}
$$ with $r_i \to \infty.$  This sequence subconverges in the pointed Gromov-Hausdorff limit to Euclidean space. Thus the limit space observes $\alpha_Y > \alpha(2,4)$ while the elements of the sequence clearly do not. 
\\

Another example where the volume growth requirement is attained in the limit but fails throughout the sequence is described below. Note that strict bounds on $\alpha_Y$ alone do not force simplified topology in the limit. Despite the very large volume growth in the limit of the following example, it is not simply connected as the elements within the sequence do not have nonnegative Ricci curvature.  

\begin{ex}\label{Ex-hyperboloid}
 (Dependence on Ricci curvature lower bound throughout the sequence)
\end{ex}

\noindent There is a sequence of complete noncompact Riemannian manifolds $M^2_i$ with linear volume growth (i.e. $\alpha_{M_i} =0$) whose limit $Y$ has larger than Euclidean volume growth but is not simply connected--owing to the lack of a Ricci lower bound for the $M^2_i$.
\\

Let $M^2$ be the one-sheeted hyperboloid in $\R^3$ defined by
\begin{equation}
M^2 = \Big\{(x,y,z) \in \R^3 : \dfrac{x^2}{a^2} + \dfrac{y^2}{b^2} - \dfrac{z^2}{c^2} = 1 \Big\}.
\end{equation} 
Let $p$ be the point $(a, 0 , 0)$ and define the manifold $M^2_i$ so that its metric $g_{M^2_i}$ agrees with that of the hyperboloid inside a ball of radius $R_i$ centered at $p$. Outside that ball, the metric $g_{M^2_i}$ is defined as the metric of the cylinder of radius $1 + \dfrac{R_i^2}{c^2}$. The two metrics can be smoothed together appropriately so that the final metric of $M^2_i$ is smooth. 

Note that $\alpha_{M^2_i} = 0$ for all $i$ because in the limit $M^2_i$ approximates the cylinder. However, this sequence of manifolds converges in the pointed Gromov-Hausdorff limit to the hyperboloid which has volume growth larger than that of Euclidean space, i.e. $\alpha_{M^2} > 1$. Naturally, the sequence does not satisfy the necessary Ricci bound and clearly the one-sheeted hyperboloid is not even simply connected regardless of how large the volume growth in the limit is.
\\

We end by giving a similar example which illustrates the necessity of $\Ric_{M_i} \geq 0$ throughout the elements of the sequence. 

\begin{ex}\label{Ex-Hawaii ring} (Large volume growth without Ricci curvature bound)
\end{ex}

Let $X \subset \R^2$ be the Hawaiian earring defined by
$$
X = \bigcup^{\infty}_{k = 1} C_k,
$$
where each $C_k$ is the circle with center $(2^{-k}, 0)$ and radius $2^{-k}$. Let $X_i = \bigcup^i_{k = 1} C_k$ and consider the product $M^2_i = \R \times X_i$ with the warped product metric
$$
(M^2_i, g_{M^2_i}) := (\R \times X_i, dt^2 + \sinh^2(t) g_{X_i}),
$$
where $g_{X_i}$ is the metric of $X_i$ induced from $\R^2$.

The sequence $\{(M^2_i, (0,0,0))\}$ converges in the pointed Gromov-Hausdorff limit to $\R \times_{\sinh^2(t)} X$ and the warping function $\sinh$ gives each element of the sequence very large volume growth. Thus, the limit has large volume growth as well. However the limit is clearly not simply connected as the Hawaiian earring is not simply connected. So, without the bound on Ricci curvature, the limit may not have trivial homotopy group even if $\alpha_Y > \alpha(1,2)$.

\section{Generalizations of Main Lemmas} \label{Section-Generalizations-Lemmas}

Recall from \cite{Munn},
\begin{Def}For constants $c>1$, $\epsilon >0$ and $n \in \mathbb{N}$, define
\begin{equation}\label{Def--gamma}
\gamma(c,\epsilon, n) = \left[1 + \left(\frac{c}{\epsilon} \right)^n \right]^{-1}.
\end{equation}
\end{Def}

\section{Generalization of Perelman's Lemma} \label{Section-Generalizations-Perelman}
\begin{Lem} \label{Prop-General Perelman}
Let $(Y, p)$ be the pointed metric measure limit of a sequence of Riemannian manifolds $\{(M^n_i, p_i)\}_{i=1}^{\infty}$ with $\Ric_{M^n_i} \geq 0$ and assume that $\alpha_Y > 1 - \gamma(c_1, \epsilon, n)$ for some $c_1 >1, \epsilon >0$. Then for any $a \in B_{p}(R)$, $R>0$, there exists a point $q \in Y \setminus B_{p}(c_1R)$ such that $d_{Y}(a, \overline{pq}) \leq \epsilon R$, where $\overline{pq}$ is a minimizing limit geodesic in $Y$ connecting $p$ and $q$.
\end{Lem}

The bound on $\alpha_Y$ indicates the sequence is noncollapsing and thus the measures do not need to be renormalized and the volumes of balls in $M^n_i$ converge to balls of the same radius in $Y$.

\begin{proof}
Let $a \in B_{p}(R)$ and choose $\delta >0$ such that $B_{a}(\delta) \subset B_{p}(R)$. By property (3) following Definition \ref{Def-pted GH convergence}, choosing $\eta < \delta /2$, we have, for $i > N_{\eta}$,
\begin{equation}
T_{\eta}(f_i(B_{p_i}(R))) \supset B_{p}(R-\eta).
\end{equation}
Since clearly $a \in B_{p}(R - \delta /2) \subset B_{p}(R - \eta) $, we have $a \in T_{\delta / 2}(f_i(B_{p_i}(R)))$ for $i$ sufficiently large. Letting $\eta \downarrow 0$, we can construct a sequence of points $a_i \in B_{p_i}(R)$ and maps $f_i: B_{p_i}(r) \to Y$ such that $f_i(a_i) \rightarrow a \in Y$. Therefore, $a \in Y$ (and in fact any point in $Y$) can be realized as the limit of a sequence of points in $M^n_i$.

Ultimately, we would like to use Perelman's Maximal Volume Lemma on elements of the limiting sequence to show that the same result holds on the limit space. However, it is possible that the manifolds in the sequence $\{(M^n_i, p_i)\}$ are compact and converge in the metric measure sense to a noncompact $(Y,p)$. With this in mind, it is necessary to appeal to a more general form of Perelman's Maximal Volume Lemma as proved in his original paper \cite{Pe}. With everything else remaining the same, the original statement assumes only $\Vol(B_p(c_2 R)) \geq (1-\gamma)\omega_n r^n$, for some $c_2 > c_1 > 1$, rather than a universal bound on the volume growth. The same proof (see \cite{Munn}, Lemma 1.5) holds with neglecting the final step of allowing $c_2$ to tend to infinity.

By Theorem \ref{Theorem-ChCo limit measure}, for $i$ sufficiently large the volume of balls $B_{p_i}(r) \subset (M^n_i, p_i)$ can be approximated by the volume of balls of the same radius in the limit space $(Y, p)$. That is to say, for any $\varepsilon > 0$, there exists an $N >0$ such that $|\nu(B_p(r)) - \Vol_{M_i}(B_{p_i}(r))| < \varepsilon$ for all $i > N$. Since, $\alpha_Y > 1-\gamma(c_1, \epsilon, n))$ and $\frac{\nu(B_p(r))}{\omega_n r^n}$ is nonincreasing as a function of  $r$, it is possible to approximate the volume of balls in the manifolds $M^n_i$ which are sufficiently close to $Y$. Namely, for constants $c_2 > c_1 > 1$ and $i$ sufficiently large,
\begin{eqnarray}
\Vol_{M_i}(B_{p_i}(c_2 R)) & > & \nu(B_{p_i}(c_2 R)) - \varepsilon\\
& > & (1 - \gamma(c_1, \epsilon, n)) \omega_n (c_2 R)^n - \varepsilon.
\end{eqnarray}

Therefore, $\Vol_{M_i}(B_{p_i}(c_2 R))  \geq (1 - \gamma(c_1, \epsilon, n)) \omega_n (c_2 R)^n$ and by Perelman's Maximal Volume Lemma, as originally stated in \cite{Pe} and described above, for each point $a_i \in B_{p_i}(R)$ there exists a point $q_i \in M_i \setminus B_{p_i}(c_1 R)$ such that $d_{M_i}(a_i, \overline{p_i q_i}) < \epsilon R$. Here $d_{M_i}$ denotes the distance function on $M^n_i$ and recall $\overline{ab}$ denotes a minimal geodesic connecting $a$ to $b$. In fact, since the points $q_i$ lie on geodesics emanating from $p_i$, it is possible to find points $q_i \in \overline{B_{p_i}(2R)} \setminus B_{p_i}(R)$ satisfying $d_i(a_i, \overline{p_i q_i}) < \epsilon R$. Again, by the properties of pointed convergence, for all $\eta >0$ and $i$ sufficiently large, there exists a map $f_i: B_{p_i}(R) \to Y$ such that 
$$
d_{GH}(B_{f_i(a_i)}(\epsilon R), B_{a}(\epsilon R)) < \eta.
$$

By controlling the location of the balls $B_{f_i(a_i)}(\epsilon R)$ in relation to the points $a, p \in Y$, it is possible to also control the location of the points $f_j(q_j)$. That is to say, for all $j > i$, the points $\{f_j(q_j)\}$ lie a compact sector of $\overline{B_{p}(2R)} \setminus B_{p}(R)$ and it is possible to extract a convergent subsequence $\{ q_{j_k} \}$ such that $f_{j_k}(q_{j_k}) \rightarrow q \in \overline{B_{p}(2R)} \setminus B_{p}(R) \subset Y \setminus B_{p}(R)$.

The limit space $Y$ is a complete length space; and thus, there exists a minimum length geodesic connecting the points $p$ and $q$, denoted $\overline{pq}$. It remains only to show that this minimal geodesic path lies within $\epsilon R$ of the point $a$. In fact, it is possible to realize this geodesic path in $Y$ as the limit of geodesics $\overline{p_i q_i}$ in $M^n_i$. Furthermore, since each of these paths lies within $\epsilon R$ of the respective points $a_i$, and the points $a_i$ are `converging' to the point $a \in Y$, the limiting geodesic path (after passing to an appropriate subsequence) must also lie with $\epsilon R$ of $a$; that is, $d_{Y}(a, \overline{pq}) \leq \epsilon R$ as required. This completes the proof.
\end{proof}

\subsection{Generalization of the Excess Estimate} \label{Section-Generalizations-AbGr}

Next, we generalize the Abresch-Gromoll excess estimate (\cite{AbGr}, c.f. \cite{Ch})  to metric measure limits of Riemannian manifolds with nonnegative Ricci curvature. In Section \ref{Section-Generalizations-Perelman} we produced a limit geodesic when proving Proposition \ref{Prop-General Perelman}. It is only necessary to prove the excess estimate for small, thin triangles which are formed from limit geodesics.

\begin{Lem}\label{Prop-General AbGr}
Let $(Y, p)$ be the pointed Gromov-Hausdorff limit of a sequence of complete Riemannian manifolds $\{(M^n_i, p_i)\}$ with $\Ric_{M^n_i} \geq 0$; $a, b \in Y$. Define, for any $x \in Y$,
$$
e_{a, b}(x) = d_Y(a,x) + d_Y(b, x) - d_Y(a, b).
$$
Set $s(x) = \min \{d_Y(a, x), d_Y(b,x)\}$ and $h(x) = d_Y(x, \overline{ab})$, where $\overline{ab}$ denotes a limit geodesic in $Y$. If $h(x) \leq s(x)/2$, then
\begin{equation}\label{excess est}
e_{a,b}(x) \leq 8 \left(
\frac{h(x)^n}{s(x)} \right)^{1/{n-1}}.
\end{equation}
\end{Lem}

\begin{proof}
Let $\epsilon >0$ and choose $0 < \eta < \epsilon/3$. Given $x, a, b \in (Y,p)$, let $x_i, a_i, b_i \in M^n_i$ be points in the sequence of manifolds that converge to $x, a, b $ (resp.) in the limit. By property (2) following the definition of pointed Gromov-Hausdorff convergence, there exists a constant $N_{\eta} >0$ such that for all $r>0$ and $i > N_{\eta}$, there is a map $f_i: B_{p_i}(r) \rightarrow Y$ such that
$$
\sup_{x_1, x_2 \in B_{p_i}(r)}|d(f_i(x_1), f_i(x_2))- d(x_1, x_2)| < \eta.
$$

\noindent This implies that, for any $\epsilon >0$,
\begin{equation}
|e_{a,b}(x) - e_{a_i, b_i}(x_i)| < 3 \eta < \epsilon,
\end{equation}
for $i$ sufficiently large, $i > N_{\eta}$. Furthermore, each element of the sequence $\{(M^n_i, p_i)\}$ has $\Ric_{M_i} \geq 0$ and so by the Abresch-Gromoll excess estimate for $M^n_i$, we find that $e_{a,b}(x) < e_{a_i, b_i}(x_i) + \epsilon \leq 8 \left(\frac{h^n(x_i)}{s(x_i)} \right)^{1/{n-1}} + \epsilon$.

Note that $s(x_i) \rightarrow s(x)$ and since we required the geodesic $\overline{ab}$ is a limit geodesic of $Y$, we also have (after passing to a subsequence if necessary) $h(x_i) \rightarrow h(x)$. Thus, for any $\epsilon' >0$,

\begin{equation}
e_{a,b}(x) < 8 \left(
\frac{h(x)^n}{s(x)} \right)^{1/{n-1}} + \epsilon'.
\end{equation} 

Since $\epsilon' > 0$ was arbitrary, (\ref{excess est}) follows and the proof is complete.
\end{proof}

\section{Proof of Theorem \ref{Thm--Main limit theorem}}\label{Section-Proof of limit theorem}
In \cite{Munn}, we use the Homotopy Construction Theorem (\cite{Munn}, Theorem 2.7) to show that $\pi_k(M^n)=0$ in Riemannian manifolds $M^n$ with nonnegative Ricci curvature and sufficiently large volume growth. In fact, the Homotopy Construction Theorem (HCT) holds for a much larger class of spaces; namely, complete, locally compact metric spaces, and thus we can also  apply it in the limit space $(Y,p)$ to show $\pi_k(Y,p)=0$. We re-state the HCT here and refer the reader to \cite{Munn} for the complete proof.

\begin{thm}{\em \textbf{(Homotopy Construction Theorem).}} \label{Theorem-HomotopyConstruction}
Let $Y$ be a complete, locally compact metric space, $p \in
Y$, $R>0$ and $f: \s^k \to B_p(R) \subset Y$ a continuous map. Given
constants $c>1$, $\omega \in (0,1)$, and a sequence of finite cell
decompositions $K_j$ of $\D^{k+1}$ with maps $f_j: \skel_k(K_j) \to Y$ satisfying
the following three properties\\
(A) $K_{j+1}$ is a subdivision of $K_j$ and $f_{j+1}
\equiv f_j$ on $K_j$ and $\max\{\diam(\sigma) | \sigma \in K_j\} \to 0$,\\
(B) For each $(k+1)$-cell, $\sigma \in K_j$, there exists a point
$p_{\sigma} \in B_p(cR) \subset Y$ and a constant $R_{\sigma} >0$
such that
$$f_j(\del \sigma) \subset B_{p_{\sigma}}(R_{\sigma});$$
and, if $\sigma' \subset \sigma$, where $\sigma'\in K_{j+1}$,
$\sigma \in K_j$, then $$B_{p_{\sigma'}}(cR_{\sigma'}) \subset
B_{p_{\sigma}}(cR_{\sigma}), \quad \textrm{and }~ R_{\sigma'} \leq
\omega R_{\sigma}, \textrm{for } \omega \in(0,1).$$ (C)
$\skel_k(K_0) = \s^k = \del \D^{k+1}$, $p_{\sigma_0}=p$, and
$R_{\sigma_0} = R$,\\ then the map $f$ can be continuously extended to a map $g:\D^{k+1} \to B_p(cR) \subset Y$. 
\end{thm}

To apply the Homotopy Construction Theorem in the limit space $(Y,p)$, we must describe a sequence of cell decompositions $K_j$ and maps $f_j$ defined on the $k$-skeletons of $K_j$ which satisfy the conditions (A), (B) and (C) above. To define $K_j$, we can use the same cell decompositions of $\D^{k+1}$ as were used in the Riemannian case. The only subtlety arises in creating the maps $f_j$. In the Riemannian case, these are constructed using a Moving In Lemma (\cite{Munn}, Lemma 3.5; \cite{Pe}, Statement (C)). As stated in \cite{Munn} and \cite{Pe}, this lemma requires nonnegative Ricci curvature and a volume growth lower bound in $M^n$ in order to apply Perelman's Maximal Volume Lemma and the Abresch-Gromoll excess estimate. However, since we have generalized versions of these lemmas for the limt space $(Y, p)$, we can prove an analog of the Moving In Lemma for the limit space $(Y, p)$.


To state the Moving In Lemma for Gromov-Hausdorff limits, recall the definition for the volume growth constant $\beta(k,c,n)$ which we defined in \cite{Munn}.

\begin{Def} \label{Def-beta volume growth}
For constants, $c>1$ and $k,n \in \N$, the value of $\beta(k,c,n)$
represents a minimum volume growth necessary to guarantee that any
continuous map $f: \s^k \to B_p(R)$ has a continuous extension $g:
\D^{k+1} \to B_p(cR)$. Define
\begin{eqnarray}
\beta(k,c,n) &= \max\{&1-\gamma(c,h^{-1}_{k,n}(c), n);
\label{Def-AbGr for
beta}\\
&& \beta(j, 1 + \frac{h^{-1}_{k,n}(c)}{2k}, n), j=1,..,k-1
\label{Def-Main(k-1) for beta}\},
\end{eqnarray}
where $\beta(0,c,n)=0$ for any $c$ and $\beta(1,c,n) = 1-\gamma(c,
h^{-1}_{1,n}(c),n)$.
\end{Def}

\noindent The definition of $\gamma(c,h^{-1}_{k,n}(c), n)$ is given in (\ref{Def--gamma}) and the function $h_{k,n}$ is defined in Section 3 of \cite{Munn}. Recall, the function $h_{k,n}$ is a smooth, one-to-one, onto increasing function which relates the constant $c>1$ and a small constant $d_0>0$. The constant $c>1$ denotes the location of the homotopic extension coming from the HCT and the constant $d_0$ describes the location of the image of the map we achieve from the Moving In Lemma. More is said about the nature of these two constants and how they are related in the discussion in Section 3 of \cite{Munn}. Given $k, n \in \mathbb{N}$, set $h_{k,n}(d_0) = c$. The coefficients of the function $h_{k,n}$ are defined iteratively and we verify that these coefficients (and thus the function as well) are optimal in the Appendix of \cite{Munn}.

\noindent We have

\begin{Pro} \emph{\textbf{(Moving In Lemma for GH limit)}.} \label{Lemma-Moving In for limit}
Let $(Y,p)$ be the pointed Gromov-Hausdorff limit of a sequence of Riemannian manifolds $\{(M^n_i, p_i)\}$ 
with $\Ric_{M_i} \geq 0$ and let $\nu$ denote the renormalized limit measure of $Y$. For a small constant $d_0 >0$ and $k,n \in \mathbb{N}$, if
\begin{equation}
 \alpha_Y \geq \beta(k, h_{k,n}(d_0), n)
\end{equation}
then given $q \in (Y,p)$, $\rho >0$, a continuous map $\phi: \s^k
\rightarrow B_q(\rho)$ and a triangulation $T^k$ of $\s^k$ such that
$\diam(\phi(\Delta^k)) \leq d_0 \rho$ for all $\Delta^k \in T^k$,
there exists a continuous map $\tphi:\s^k \rightarrow
B_q((1-d_0)\rho)$ such that
\begin{equation}\label{diam conclusion from moving in}
\diam(\phi(\Delta^k) \cup \tphi(\Delta^k)) \leq 10^{-k-1} \left(1+
\frac{d_0}{2k}\right)^{-k} (1-h_{k,n}(d_0)^{-1}) \rho.
\end{equation}
\end{Pro}


\noindent {\em Sketch of proof.}  The original idea and proof of the Moving In Lemma for Riemannian manifolds is due to Perelman (\cite{Pe}, Statement C). In \cite{Munn} we extend Perelman's result to determine precise constants for the volume growth which describe how varioius homtopy dimensions are influenced as the volume growth increases. The analysis to determine these precise bounds for the volume growth of the Riemannian manifold is given in great detail in \cite{Munn}. Here we show that the result of the Moving In Lemma can be extended to Gromov-Hausdorff limits assuming a similar volume growth bound is obeyed in the limit $(Y,p)$.

The proof of the Moving In Lemma for GH limits mirrors the proof of the Moving In Lemma for Riemannian manifolds which can be found in Section 3 of \cite{Munn}. The proof is by construction and the map $\tphi$ is built inductively on $i$-dimensional skeletons of the given triangulation of $\s^k$. The key point in the Riemannian case which requires the smooth structure arises in an application of Perelman's Maximal Volume Lemma to create a long well-placed geodesic in $M^n$. One then applies the Abresch Gromoll excess estimate to the long thin triangle made from this geodesic. The proof of the Moving In Lemma for the limit space $(Y,p)$ follows verbatim replacing the original Perelman Maximal Volume Lemma with our generalized version (Lemma \ref{Prop-General Perelman}) and replacing the Abresch-Gromoll excess estimate with our generalized excess estimate as applied to limit geodesics in $(Y,p)$ (Lemma \ref{Prop-General AbGr}). This completes the proof of the Proposition. \hfill $\square$
\\

Before we proceed to the proof of Theorem \ref{Thm--Main limit theorem}, let us recall Perelman's argument describing how to apply the Moving In Lemma to create a homotopic extension of a continuous map $f$ on $\s^k$. Ultimately, this amounts to an application of the Homotopy Construction Theorem. The hypothesis of the Homotopy Construction Theorem requires a sequence of cell decompositions $K_j$ of $\D^{k+1}$ and a sequence of maps $f_k$ defined on the $k$-skeletons of the $K_j$ which satisfy the conditions (A), (B), and (C). Following \cite{Pe}, take any $(k+1)$ cell in $K_j$ and express it in polar coordiates as $\s^k \times (0,1] \cup \{0\}$. Let $T^k$ be a triangulation of $\s^k$ (satisfying the condition of the Moving In Lemma) and decompose this cell into components so that the original cell intersects the $k$-skeleton of the decomposition at 

\begin{equation}\label{cell decomp}
 \s^k \times \{1\}~ \cup~ \s^k \times \{\dfrac{1}{2}\}~ \cup ~ \skel_{k-1}(T^k) \times [\dfrac{1}{2}, 1].
\end{equation}

This process can be repeated on smaller and smaller scales so that as $j \to \infty$ the $k$-skeleton of $K_j$ creates a very fine net filling in $\D^{k+1}$. We use the Moving In Lemma to define a sequence of continuous maps $f_j$ on the $\skel_j(K_j)$ for each $j$. We must define $f_{j+1}$ on the three components of (\ref{cell decomp}) above. Set $f_{j+1} \equiv f_j$ on $\s^k \times \{1\}$. Applying the Moving In Lemma (taking $f_{j}$ for $\phi$), set $f_{j+1} = \tilde{f_j}$ on $\s^k \times \{\dfrac{1}{2}\}$. Lastly, using induction and assuming that lower dimensional maps can be extended across $\s^{i}$ for $i = 0, ..., k-1$, we can fill in the maps $f_{j+1}$ consecutively across $\skel_i(T^k) \times [1/2, 1]$ for $i = 0, ..., k-1$. This construction of cell decompositions $K_j$ and maps $f_j$ satisfies the hypothesis of the Homotopy Construction Theorem and thus a homotopic extension of the map $f$ exists.

We can use precisely the same idea in $(Y,p)$ to construct a homotopic extension in the limit space as well. The only necessary tools were the Moving In Lemma and the Homotopy Construction Theorem. Since the Homotopy Construction Theorem holds for any complete, locally compact metric space and Proposition \ref{Lemma-Moving In for limit} provides an analog of the Moving In Lemma for GH limits, we can now prove Theorem \ref{Thm--Main limit theorem}.
\\

\begin{proof}
As in \cite{Munn}, we want to show that a $k$-sphere is contractible in the limit space. We proceed by repeatedly applying the Moving In Lemma for GH limits keeping track of the constants until we satisfy the conditions of the Homotopy Contstruction Theorem. Since the Homotopy Construction Theorem holds for any complete locally compact length space, it certainly holds in our limit space $(Y,p)$. Furthermore, the values we found for $\alpha_M$ which guarantee $\pi_k(M^n) = 0$ in the Riemannian setting are precisely the bounds necessary to meet the hypothesis of the Moving In Lemma for GH limits and thus ensure $\pi_k(Y) = 0$ for the limit space. This proves Theorem \ref{Thm--Main limit theorem}.
\end{proof}

In \cite{Munn} we extend Perelman's work by carefully analyzing the nature of the expression $\beta(.,.,.)$ to determine explicit values for the $\alpha(k,n)$ of Theorem \ref{Thm--Main limit theorem}. To determine the optimal bound for the volume growth (as determined via this method) set
\begin{equation}
 \alpha(k,n) = \inf_{c \in (1, \infty)} \beta(k,c,n).
\end{equation}

The expression for $\beta(k,c,n)$ is iterative and the number of terms in the maximum increases as $2^{k-1}$. However, we verify in \cite{Munn} that the leading term $1 - \gamma(c, h^{-1}_{k,n}(c), n)$ in fact dominates the maximum of the collection. Thus, by examining the behavior of $\gamma(c, h^{-1}_{k,n}(c), n)$ as a function of $c$, we are able to to extract precise constants $\alpha(k,n)$ for $\alpha_M$ which guarantee the $k$th homotopy group of $M^n$ is trivial. The same constants produce an equivalent outcome for the homotopy groups of the GH limit space $(Y,p)$.

In the table below we indicate the values for $\alpha(k,n)$ for $1 \leq k \leq 3,~ 1 \leq n \leq 10$. The explicit form for higher dimensional $\alpha(k,n)$ can be found in \cite{Munn}.

\begin{table}[!h]
\caption{\small{Table of $\alpha(k,n)$ values for $1 \leq k \leq 3, 1 \leq n \leq 10$}}
\label{Table-alphas}
\begin{center}
\begin{tabular}{|c||c|c|c|}
  \hline
   & $k=1$ & $k=2$ & $k=3$ \\ \hline \cline{1-4}
  $n=1$ & $1 - 1.04 \times 10^{-5}$ & - & - \\ \hline
  $n=2$ & $1 - 1.65 \times 10^{-14}$ & $1 - 7.05 \times 10^{-44}$ & - \\ \hline
  $n=3$ & $1 - 3.95 \times 10^{-28}$ & $1 - 1.13 \times 10^{-91}$ & $1 - 2.06 \times 10^{-289}$ \\ \hline
  $n=4$ & $1 - 3.02 \times 10^{-46}$ & $1 - 1.23 \times 10^{-178}$ & $1 - 9.30 \times 10^{-734}$ \\ \hline
  $n=5$ & $1 - 7.46 \times 10^{-69}$ & $1 - 5.61 \times 10^{-309}$ & $1 - 9.16 \times 10^{-1583}$ \\ \hline
  $n=6$ & $1 - 5.94 \times 10^{-96}$ & $1 - 1.01 \times 10^{-491}$ & $1 - 2.57 \times 10^{-3035}$ \\ \hline
  $n=7$ & $1 - 1.53 \times 10^{-127}$ & $1 - 6.66 \times 10^{-736}$ & $1 - 1.18 \times 10^{-5330}$ \\ \hline
  $n=8$ & $1 - 1.27 \times 10^{-163}$ & $1 - 1.50 \times 10^{-1050}$ & $1 - 2.33 \times 10^{-8748}$ \\ \hline
  $n=9$ & $1 - 3.40 \times 10^{-204}$ & $1 - 1.07 \times 10^{-1444}$ & $1 - 2.28 \times 10^{-13608}$ \\ \hline
  $n=10$ & $1 - 2.95 \times 10^{-249}$ & $1 - 2.24 \times 10^{-1927}$ & $1 - 5.70 \times 10^{-20271}$ \\ \hline
\end{tabular}
\end{center}
\end{table}

\vspace{5mm}
\singlespacing
\noindent
Michael Munn \\
Mathematics Institute\\
University of Warwick\\
Coventry, UK\\
e-mail: mikemunn@gmail.com

\end{document}